\documentclass[12pt,twoside]{article}
\makeatletter
\evensidemargin \oddsidemargin
\usepackage{latexsym} 
\usepackage{theorem}
\usepackage[active]{srcltx}
\newtheorem{theorem}{Theorem}[section]
\newtheorem{lemma}[theorem]{Lemma}
\newtheorem{corollary}[theorem]{Corollary}
\newtheorem{proposition}[theorem]{Proposition}
\newtheorem{conjecture}[theorem]{Conjecture}
{\theorembodyfont{\rmfamily}  	
				\newtheorem{remark}[theorem]{Remark}
				\newtheorem{question}[theorem]{Question}
}
{\theoremstyle{break}	
{\theorembodyfont{\rmfamily} 
				\newtheorem{example}[theorem]{Example}
				\newtheorem{hypothesis}[theorem]{Hypothesis}
}}
{\end{enumerate}}

\def\qed{\ifhmode\unskip\nobreak\fi\ifmmode\ifinner\else\hskip5pt\fi\fi
 \hfill\hbox{\hskip5pt\vrule width4pt height6pt depth1.5pt\hskip1pt}}

\newenvironment{proof}[1]{\smallskip \noindent {\bf #1}}{\qed\smallskip}

\def\bc{\begin{center}}
\def\ec{\end{center}}

\def\bthm{\begin{theorem}}
\def\ethm{\end{theorem}}

\def\bcor{\begin{corollary}}
\def\ecor{\end{corollary}}

\def\bprop{\begin{proposition}}
\def\eprop{\end{proposition}}

\def\blem{\begin{lemma}}
\def\elem{\end{lemma}}

\def\bpr{\begin{proof}}
\def\epr{\end{proof}}

\def\bex{\begin{example}}
\def\eex{\end{example}}

\def\brem{\begin{remark}}
\def\erem{\end{remark}}

\def\bhyp{\begin{hypothesis}}
\def\ehyp{\end{hypothesis}}

\def\bques{\begin{question}}
\def\eques{\end{question}}

\def\bconj{\begin{conjecture}}
\def\econj{\end{conjecture}}

\def\bdes{\begin{description}}
\def\edes{\end{description}}
\def\bitm{\begin{itemize}}
\def\eitm{\end{itemize}}

\def\ita{\item[(a)]}
\def\itb{\item[(b)]}

\def\itd{\item[(d)]}

\def\iti{\item[(i)]}
\def\itii{\item[(ii)]}
\def\itiii{\item[(iii)]}
\def\itiv{\item[(iv)]}
\def\itv{\item[(v)]}
\def\itvi{\item[(vi)]}

\def\beq{\begin{equation}}
\def\eeq{\end{equation}}

\def\ben{\begin{enumerate}}
\def\een{\end{enumerate}}

\def\beqar{\begin{eqnarray}}
\def\eeqar{\end{eqnarray}}
\def\beqarr{\begin{eqnarray*}}
\def\eeqarr{\end{eqnarray*}}

\def\R#1{{\bf R}^{#1}}   
\def\NN{{\bf N}}	   
\def\Np{{\bf N}_{+}}	   

\def\part{\partial}

\def\d#1dt{\frac{d#1}{dt}}    

\def\thl#1{\mylabel{th:#1}} 
\def\secl#1{\mylabel{sec:#1}} 

\def\rth#1{\ref{th:#1}}    		
\def\rsec#1{\ref{sec:#1}} 		 

\def\eps{\epsilon}

\def\del{\delta}

\def\ga#1{\gamma(#1)}

\def\oc#1{\ensuremath{\overline \gamma}(#1)}    
\def\om#1{\ensuremath\omega(#1)}

\def\Int{\mathrm{Int}\,}
\def\clos{\mathrm{clos}\,}

\def\Fix#1{{\mathrm}}

\def\Fix#1{{\mathrm Fix} (#1)}

\def\Int{\mathrm {Int}}
\def\Ep{E^+}
\def\Epp{\Int\,\Ep}

\def\clos{\mathrm{clos}\,}
\def\dist{\mathrm{dist}\,}

\def\supp{\:\mathrm{Supp}\:}
\def\SS{{\mathbf S}}
\def\oc#1{{\overline \gamma}(#1)}    
\def\om#1{\omega(#1)}

\def\mylabel#1{\label{#1}}

\reversemarginpar   
\def\oc#1{{\overline \gamma}(#1)}    
\def\om#1{\omega(#1)}

\begin{document}

\renewcommand\section{\@startsection {section}{1}{\z@}%
                                   {-3.5ex \@plus -1ex \@minus -.2ex}%
                                   {2.3ex \@plus.2ex}%
                                   {\normalfont\large\bfseries}}
\renewcommand\subsection{\@startsection{subsection}{2}{\z@}%
                                     {-3.25ex\@plus -1ex \@minus -.2ex}%
                                     {1.5ex \@plus .2ex}%
                                     {\normalfont\bfseries}}

$~$\\[.5in]
\begin{center}
{\large
\bf 
CHAIN TRANSITIVE SETS FOR SMOOTH STRONGLY MONOTONE
DYNAMICAL SYSTEMS}
\vskip.25in
  Morris W. Hirsch$^{1}$\\[2mm]
{\footnotesize $^1$Department of Mathematics\\University of California at 
Berkeley}
\end{center}

\medskip
{\footnotesize 
\noindent
{\bf Abstract.} Let $K$ denote a compact invariant set for a strongly
monotone semiflow in an ordered Banach space $E$, satisfying standard
smoothness and compactness assumptions. Suppose the semiflow restricted to
$K$ is chain transitive.  The main result is that either $K$ is
unordered, or else $K$ is contained in totally ordered, compact arc of
stationary points; and the latter cannot occur if the semiflow is real
analytic and dissipative.  As an application, entropy is $0$ when
$E=\R 3$.  Analogous results are proved for maps.  The main tools are
results of Mierczy\'nski \cite{Mi94} and Tere\v{s}\v{c}\'ak
\cite{Te94}.}

\smallskip
\noindent{\bf AMS subject classfication:} 37C65

\section{Introduction}\mylabel{sec:intro}

Throughout this paper $E$ denotes an ordered Banach space whose order
cone $\Ep=\{x\in E:x\ge 0\} $ has nonempty interior $\Epp$.
Our main results are Theorem  \rth{main} concerning smooth strongly monotone
maps, and its analogue for semiflows, Theorem  \rth{semi}. 

\paragraph{Notation and terminology.}
 We write $x\ge y$ if $x-y\in \Ep$ and
$x>y$
if $x\ge y$ and $x\ne y$, or equivalently, $x-y\in \Ep \setminus
\{0\}$. If $x-y\in \Epp$ we write $x\gg 0$.

 A subset $Y\subset X$ is {\em unordered} if none of its points satisfy
$x<y$.  We call $x$ and $y$ {\em comparable} in case $x\le y$ or $y\le x$.

If $A$ and $B$ are sets, notation such as $A>B$
means $a>b$ for all $a\in A, b\in B$.  We abuse notation slightly by
writing $A>b$ if $A$ is a set, $b$ is a point, and $A>\{b\}$. 
For any  subsets or points $A, B$   
\[A_+=\{x\in E: x\ge A\}, \quad\quad
A_-=\{x\in E: x\le A\} \]
\[ [A,B]=\{x\in E:A\le x\le B\}\]
\[ [[A,B]]= \{x\in E:A\ll  x\ll B\}\]
For  sets $A, B$ 
\[A^\bullet=\bigcup_{x\in A}x_+,\quad\quad
A_\bullet=\bigcup_{x\in A}x_-  \]
\[\lfloor A, B\rceil= A^\bullet\cap B_\bullet=
  \bigcup_{a\in A,b\in B}[a,b]\]
If $ A\subset B_\bullet$ then {\em $B$ majorizes $A$}, while if $A\subset
B^\bullet$ then {\em $B$ minorizes $A$}.  

$S:X\to X$ always denotes a continuous map defined in an open set
$X\subset E$.  The {\em orbit} $\ga x$ of $x\in X$ is the set $\{S^n
x:n\in \NN\}$ where $\NN=\{0,1,\ldots\}$ is the set of natural
numbers.  The {\em orbit closure} of $x$ is $\oc x =\clos \ga x$.  The
{\em omega limit set} of $x$ is $\om x=\bigcap_{k\ge 0} \oc{S^k x}$.

A point $p$ is {\em $m$-periodic} if $S^m p=p$.  If also $\om z =\ga
p$ for some $z>p$ then $p$ and $\ga p$ are {\em upper
attracting}.  A {\em lower attracting} periodic point is defined
dually (by reversing order relations).

The following hypotheses is always in force:
\begin{hypothesis}\thl {hyp0}
\begin{description}
\iti $X\subset E$ is an open set that is {\em order convex}: if $p,q
\in X$ then $X$ contains the closed order interval $[p,q]$.

\itii $S:X\to X$ is $C^1$ (continuously Frechet
differentiable).

\itiii 
$S$ is  {\em order compact}: the image of each closed order
interval $[p,q],\:p,q\in X$, has compact closure in $X$.

\itiv For each $x \in X$ the derivative of $S$ at $x$ is a {\em
strongly positive} linear operator $DS (x):E\to E$, i.e., $DS(x)u\gg
0$ if $u>0$.  This implies $S$ is {\em strongly monotone}: $Sx \gg Sy$
if $x>y$.

\itv  $S$ is {\em completely continuous}: the image of
any bounded set has compact closure in $X$.  Moreover, each derivative
$DS(x)$ is completely continuous. 

\itvi Every orbit closure is a compact subset of $X$.
\end{description}
\end{hypothesis} 
Such maps arise in ordinary, partial and functional differential
equations; see e.g.\ Dancer \& Hess \cite{DH91}, Hess \& Pol\'a\v{c}ik
\cite{HP93}, Hirsch \cite{Hi88, Hi94}, Pol\'a\v{c}ik \&
Tere\v{s}\v{c}\'ak \cite{PT91, PT93}, Smith \cite{Sm95}, Smith \&
Thieme \cite{ST91a}, Tak\'a\v{c} \cite{Ta90, Ta92}.

Our results are based on the following fundamental  property of this class
of maps (which does not require order compactness):
\begin{theorem}[\sc Tere\v{s}\v{c}\'ak \cite{Te94}]	
\mylabel{th:gen}
There exists a natural number $m\ge 1$ and an open dense set of points
$x$ such that $\om x$ is an $m$-periodic orbit.
\end{theorem}
Similar results under  stronger hypotheses have been obtained
 by  Tak\'a\v{c} \cite{Ta90}, Pol\'a\v{c}ik \& Tere\v{s}\v{c}\'ak
 \cite{PT91},  Hess \&  Pol\'a\v{c}ik \cite{HP93}. 

 The smallest number $m$ satisfying Theorem \rth{gen} is called the {\em
generic asymptotic period}.
 
\paragraph{Attractor-free sets, p-arcs and the main theorem for maps.}
Let $T:Y\to Y$ be a continuous map.  A subset $A\subset Y$ is an
{\em
attractor} if $A$ is invariant ($T (A)=A$) and contained in an open
set  $W\subset Y$ such that
\[\lim_{n\to\infty}\dist (T^n w, A)=0 \quad\mbox{uniformly in $w\in
W$}.\] If also  $A\ne Y$ then $A$ is a {\em proper attractor}.  $T$ is
{\em dissipative} if there is an attractor that contains all omega
limit points.  

Now let $Y\subset X$ be a compact invariant set for $S$.  If the map
$S|Y$ does not have a proper attractor, $Y$ is {\em attractor-free}.
By a theorem of Conley \cite{Co78} this is equivalent to $S|Y$ being
chain transitive (see Section \rsec{prelim}).  Every omega limit set
is attractor-free and unordered.  Attractor-free sets for semiflows
occur as limit sets of several kinds of dynamic and stochastic
processes (Bena\"{\i}m and Hirsch \cite{BH95, BH96, BH97}, Mischaikow
{\em et al.\ }\cite{MST95}, Thieme \cite {Th92, Th94}).
  
A {\em smooth arc} $J\subset E$ is the homeomorphic image of the
closed unit interval under an injective $C^1$ map $h:[0,1]\to E$ whose
derivative is nonzero everywhere.  We call such a $J$ a {\em p-arc}
for $S$ if $S(J)=J$ and $h'(t)\gg 0$ for all $t\in [0,1]$.  This makes
$J$ totally ordered.\footnote
{The concept of p-arc is due to Mierczy\'nski \cite{Mi94}, who also
allowed degenerate p-arcs, i.e., fixed points.  Here we allow only
nondegenerate p-arcs.}
The endpoints $h(0),\,h (1)$ are fixed points.  The set of endpoints
is denoted by $ \partial J$.

A p-arc is {\em stationary} if it consists of fixed points.  It is
easy to see that if $J$ is a stationary p-arc for $S^m$, then each
point of $J$ is $m$-periodic,  there is a divisor $d\ge 1$ of $m$
such that the arcs $J, S(J),\ldots, S^{d-1}(J)$ are disjoint and
permuted cyclically by $S$, and if $0\le i<j\le d-1$ then no point of
$S^iJ$ is related to any point of $ S^j J$.

We can now state our main results for maps.  The following two
theorems refer to a map $S:X\to X$ satisfying Hypotheses \rth{hyp0},
with generic asymptotic period $m\ge 1$. 

\begin{theorem} \mylabel{th:main}  
 Let $K\subset X$ be an
attractor-free set.  Then either
\begin{description}
\ita  $K$ is unordered; or else

\itb $K=\bigcup_{0\le i<m} S^i (J)$ where $J$ is a stationary p-arc
for $S^m$.
\end{description}
Moreover 
$K$ is unordered provided $S$ is real analytic and dissipative.
\end{theorem} 

The following result complements shows that except for upper
attracting $m$-periodic orbits, unordered attractor-free sets are
rather unstable:
 
\begin{theorem}		\mylabel{th:BH}
Let $K$ be an unordered attractor-free set.  Then either $K$ is an
upper attracting $m$-periodic orbit, or else
 there exists a lower attracting   $m$-periodic  point $q$  such that
$\ga q$ majorizes $K$ and $K$ minorizes $\ga q$, and
$\om y=\ga {q}$ if $x<y<u$ for some $x\in K, u\in \ga q$. 
The dual result also holds.
\end{theorem}
Thus if $K$ is not an upper attracting $m$-periodic orbit,
it lies in the upper boundary of the basin of attraction of $\ga {q}$.
\paragraph{The main theorem for semiflows.}   \mylabel{sec:semi}
Let ${\mathbf S} =\{S_t:X\to X\}_{o\le t<\infty}$ be a {\em semiflow}
in $X$, i.e., $S_tx$ is continuous in $(t,x)$, $S_t\circ S_r=S_{t+r}$,
and $S_0$ is the identity map of $X$.  A point $p$ is an {\em
equilibrium} if $S_tp=p$ for all $t\ge 0$.  We always assume:

\begin{hypothesis}\thl {hyp1}
 Each map $S_t:X\to X,\: t>0$ satisfies Hypothesis \rth{hyp0}.
\end{hypothesis}

Let $Y\subset X$ be invariant under $\mathbf S$, i.e., $S_t(Y)=Y$ for
all $t$.  An {\em attractor} for the restricted semiflow 
${\mathbf S}|Y=\{S_t|Y:Y\to Y\}$ 
is a nonempty compact invariant $A\subset Y$
having a neighborhood $V\subset Y$ such that
$\lim_{t\to\infty}\dist (S_ty, A)=0$ uniformly in $v\in
V$.  The definitions for semiflows of  proper attractor, 
attractor-free, and dissipative are similar to those for maps. 

A smooth arc $J\subset X$ is a {\em p-arc for $\mathbf S$} if it is a
p-arc for every map $S_t$.  If in addition every point of $J$ is an
equilibrium, then $J$ is a {\em stationary} p-arc for $\SS$.

The analogue of Theorem  \rth{main} for semiflows is:

\begin{theorem}		\mylabel{th:semi}
Assume the semiflow $\mathbf S$ satisfies Hypothesis \rth{hyp0} and
let $K\subset X$ be an attractor-free set.  Then either $K$ is
unordered, or $K$ is a stationary p-arc for $\mathbf S$.  If $\mathbf
S$ is dissipative and real analytic, then $K$ is unordered.
\end{theorem}
This result has been applied to stochastic approximation and game
theory in Bena\"{\i}m \& Hirsch \cite{BH97}. 

Theorem  \rth{BH}
 takes the following form for semiflows:

\begin{theorem}		\mylabel{th:BHsemi}
Assume the semiflow $\mathbf S$ satisfies Hypothesis \rth{hyp1}, and
let $K\subset X$ be an unordered attractor-free set that is not an
upper attracting equilibrium. Then there is a lower attracting
equilibrium $q\gg K$ such that 
$ S_t y \to q$  if $x<y<q$ for some $x\in K$.
\end{theorem}

\paragraph{Application to invariant measures.}   \mylabel{sec:invmeas}
Before beginning the proof of the main theorems, we use them to
investigate invariant measures.

An {\em invariant measure} $\mu$ for $S$ is a Borel probability
measure on $X$ with compact support such that $\mu (U)=\mu (S^{-1}U)$
for every $\mu$-measurable set $U$.  The {\em support} $\supp (\mu)$ of
$\mu$ is the complement in $X$ of the union of all open sets $U$ such that
$\mu (U)=0$.  Notice that $\supp (\mu)$ is an invariant set.  If $\mu$ is
invariant for every map in a semiflow $\SS$, then $\mu$ is called an
invariant measure for $\SS$.  An invariant measure is {\em ergodic} if
every measurable invariant set has measure zero or one.

Chain recurrence is defined in Section \rsec{prelim} below.

\begin{proposition}		\mylabel{th:meas}
Let $K$ be the support of an invariant measure $\mu$.  Then every
point of $K$ is chain recurrent for $S|K$.  If $\mu$ is ergodic, $K$ is
attractor-free.
\end{proposition}
\begin{proof}
We rely on the fact that $\mu (Q)-0$ for every nonempty, relatively open set
$Q\subset \supp (\mu)$.  Let 
$R\subset K$ denotes closure of the  set of
recurrent points in $K$.  If $K\setminus R\ne \emptyset$ then $\mu
(K\setminus R)>0$.  But then $K\setminus
R$ carries an invariant measure, and thus contains a recurrent point
by Poincar\'e's recurrence theorem (Nemytskii \& Stepanov \cite{NS56},
Chapter VI, Theorem 3.05).  This contradiction shows every point of
$K$ is recurrent, and therefore chain recurrent.

 Suppose $\mu$ is ergodic and let $A\subset K$ be an attractor for
$S|K$.  Then, as is well known, every neighborhood of $A$ in $K$
contains a compact neighborhood $N\subset K$ of $A$ such that
$S(N)\subset \Int_K N$ and $\bigcap_{j\ge 0}S^j (N)=A$.
The set $N\setminus S (N)$, being open in $K$ and disjoint from its
image, has measure $0$.  Therefore $S(N)\setminus N=\emptyset$, so
that $N=A$.  Therefore both $A$ and $K\setminus A$ are invariant and
open in $\supp (\mu)$.  
Then $\mu(A)>0$ since $A$ is nonempty.
If $K\setminus A$ is nonempty it has positive measure by the same
argument; but this contradicts ergodicity.  Therefore $A=K$.
\end{proof}

The topological entropy of a dynamical system is a much studied
numerical invariant.  While the definition is too complicated to give
here (see e.g.\  Katok \& Hasselblatt \cite{KH95}), it can be noted that
positive entropy is often used as a criterion for chaos.  Conversely,  zero
entropy suggests nonchaotic behavior.  Rigorously, zero entropy
implies that the system does not contain a subsystem dynamically
equivalent to a Smale horseshoe.  

Under Hypotheses  \rth{hyp0} or  \rth{hyp1} we have the following result:
\begin{theorem}		\mylabel{th:ent1}
Assume $E=\R 3$ and $\SS$ is a flow (respectively, $E=\R 2$ and $S$ is
a diffeomorphism).  Then $\SS$ (respectively, $S$) has topological
entropy $0$.
\end{theorem}
\begin{proof}
 The topological entropy is the supremum of the measure theoretic
entropies of ergodic invariant measures (Goodwyn \cite{Go69}).
Therefore it suffices to prove that every ergodic invariant measure
$\mu$ for $\SS$ (or $S$) has measure theoretic entropy $0$.

Consider a flow $\SS$ in $X\subset \R 3$.   Then $K$ is
attractor-free by Proposition \rth{meas}, and Theorem \rth{main}
applies.  It is easy to see that ergodicity precludes conclusion (b)
of Theorem \rth{main}, so $K$ is unordered.  Therefore $K$ lies in an
invariant surface  (Hirsch \cite{Hi88a}, Tak\'a\v{c} \cite{Ta90}).
Since every surface flow has entropy $0$ (Young \cite{Yo77}), the
proof for flows is complete.

The proof for a diffeomorphism $S$ in $X\subset \R 2$ is similar:
Theorem \rth{semi} implies $K$ is unordered, $K$ lies in an invariant
$1$-manifold, and every homeomorphism of a $1$-manifold has entropy $0$
(Adler {\em et al.\ }\cite{AKM65}).
\end{proof}

\section{Preliminaries}   \mylabel{sec:prelim}

\paragraph{Chain equivalence.} 
Let  $T:Y\to Y$ be a continuous map in a metric space.
Let $u,v\in Y$.  We say $u$ {\em $\eps$-chains to $v$},  written 
$u\leadsto_\eps v$, 
if there exist a
number $m\in \Np$ (the set of positive natural numbers) and a finite
sequence in $Y$ of the form
\[u=y_0,\ldots, y_m=v\]
 such that
\[ ||T y_{i-1}- y_{i}|| < \eps,\quad i=1,\ldots,m \]
The $m+1$-tuple $(y_1,\ldots,y_n)$ is an {\em $\eps$-chain.}
If $u\leadsto_\eps v $ for every $\eps >0$ then $u$ {\em chains to
$v$}, denoted by $u\leadsto v$. 
If $u\leadsto v $ and $v\leadsto u $ then $u$ and $v$ are {\em chain
equivalent}, written $u\approx v$.  If  $u\leadsto u $  then $u$ is
{\em chain recurrent}.  
Define
\[\Omega (u, T)=\{v:u\leadsto v\}\]
This set is closed and forward invariant;  when $u$ is chain
recurrent, it is invariant. 

 The binary relation of
chain equivalence is symmetric and transitive, and reflexive on the
the set $CR(T)$ of chain recurrent points of $T$.  This closed invariant
set contains the nonwandering set, all homoclinic and heteroclinic
cycles, and all supports of invariant measures.

If every point is chain recurrent, then $T$ is a {\em chain recurrent
map}.  If $Y= CR(T)$ then we say $Y$ is {\em chain transitive}.  When
$Y$ is compact, this is equivalent to $Y$ being attractor-free by a
theorem of Conley \cite{Co78} for semiflows and its analogue for
maps. 

Now consider a semiflow $\mathbf T=\{T_t\}_{0\le t<\infty}$ in $Y$.
For $R>0$ and $\eps >0$ we say $u$ {\em $(R,\eps)$-chains to $v$},
written $u\leadsto_{R,\eps} v,$ if there exists a natural number $m\ge
1$, real numbers $t_1,\ldots,t_r\ge R$, and a finite sequence in $Y$
of the form
$v=y_0,\ldots, y_n=u$
 such that
\[||T_{t_i} y_{i-1}-y_{i}||<\eps\quad i=1,\ldots,m\]
If $u\leadsto_{R,\eps} v $ for every  $R>0,\eps >0$ then $u$ {\em chains to
$v$}, denoted by $u\leadsto v$.  If $u\leadsto u$ then $u$ is {\em
chain recurrent for $\mathbf T$.} 

The definitions of chain recurrence, chain equivalence and chain
transitivity for semiflows are  analogous to those for maps.  When  
$Y$ is compact, chain transitivity is equivalent to attractor-free.

\paragraph{Monotone convergence and p-arcs.}
We return to the map $S:X\to X$ satisfying Hypothesis \rth{hyp0}.

A point $x$ is {\em convergent} if $\om x$ is a singleton (necessarily
a fixed point).

Suppose $\om x =p$. We say $\ga x$ {\em eventually decreases to $p$}
and write $\ga x \searrow p$ provided there exists $n\ge 0$
such that $S^{n} x > S^{n+1} x$, in which case
\[S^kx \gg S^{k+1}x\gg  p\:\: \mbox{for all $k\ge n$}.\]
 If $S^{n} x < S^{n+1} x$ for some $n\ge 0$, we say $\ga x$ {\em
eventually increases to $p$} and write $\ga x \nearrow p$.  When $\ga
x \searrow p$ or $\ga x \nearrow p$, we say $\ga x$ is {\em eventually
monotone} and converges {\em eventually monotonically}.

\begin{theorem}[\sc  Mierczy\'nski]
\mylabel{th:mier1} 
Let $p$ be a fixed point.  The set of points whose orbits converge to
$p$, but not monotonically, is unordered.
\end{theorem}
\bpr Follows from Proposition 2.1 of \cite{Mi94}.
\epr

For any $x\in X$ let $\rho (x)$ denote the spectral radius of the
linear operator $L_x=DS (x):E\to E$.  The Krein-Rutman theorem
(Deimling \cite{De80}) implies $\rho (x)$ is a simple eigenvalue of
$L_x$, whose one-dimensional eigenspace, called the {\em principal
eigendirection} $E_1(x)$, is spanned by a vector $\gg 0$.  There is a
direct sum decomposition $E=E_1 (x)\oplus E_2 (x)$ invariant under
$L_x$ such that $L_x|E_2 (x)$ has spectral radius $\rho_2(x)<\rho(x)$.
Moreover $E_2 (x)\cap\Ep =\{0\}$.

It is easy to see that if $x$ belongs to a stationary p-arc then $\rho
(x)=1$.
 
\begin{theorem}
[\sc  Mierczy\'nski]\mylabel{th:parcs}
{~}
\begin{description}

\ita Assume $p,q$ are fixed points with $p<q$ and the set of fixed points
in $[p,q]$ is compact.  Then there is a p-arc whose endpoints are $p$
and $q$.

\itb If $x$ is a fixed point in a p-arc $J$ then the tangent space
$T_x J$ is the principal eigendirection  $E_1(x)$. 

\end{description}
\end{theorem}
\bpr
Part (a) is proved in Theorem 3.16 of \cite{Mi94}. Part (b) is Lemma 3.6
of  \cite{Mi94}.
\epr

Let $J$ be a stationary p-arc.  There is a continuous family of bounded,
$C^1$ hypersurfaces $\{{\cal L}(x)\}_{x\in J}$ with the following
properties (see Proposition 3.8 of \cite{Mi94}):
\blem  [\sc Mierczy\'nski]\thl{leaves}
{~}
\begin{description}
\item[(a)] ${\cal L}(x)$ is tangent to $E_2 (x)$ at
$x$

\item[(b)] ${\cal L}(x)$ is forward invariant

\item[(c)] $\lim_{n\to\infty}S^n y=x$  uniformly for $y\in {\cal L}(x)$

\itd
Set $B(J)=\bigcup_{x\in J}{\cal L}(x)$.  Then the interior of $B(J)$
is a neighborhood of $J\setminus \partial J$, and $J$ is a global
attractor for $S|B(J)$.
\end{description}
\elem
I call $B(J)$ a {\em contracting collar} for $J$.

Denote the two endpoints of $J$ by $e_0\ll e_1$.  Set
\[Q(J)=\lfloor {\cal L} (e_0), {\cal L}(e_1) \rceil\]
\begin{proposition}[\sc Mierczy\'nski]\mylabel{th:pnbd}
$B(J)$ is a neighborhood of $J$ in $Q(J)$.
\end{proposition}
\bpr This is Proposition 3.8(vi) of \cite{Mi94}. 
\epr

Define  forward invariant sets
\beqarr
V^-(J)	&=&	\{x\in X:\ga x\cap (\Int J_-) \ne\emptyset\},\\
V^+(J)	&=&	\{x\in x:\ga x\cap \Int (J_+) \ne\emptyset\},\\
V (J)&=&	X\setminus \left(V^-(J)\cup V^+(J)\right)
\eeqarr
Then $V^-(J)$ and $V^+(J)$ are open and $V (J)$ is closed in $X$.

\begin{proposition}		\mylabel{th:pnbd3}
$B(J)$ is a neighborhood of $J$ in $V (J)$.
\end{proposition}
\bpr By Proposition \rth{pnbd} it suffices to prove that $V (J)\cap Q(J)$
is a neighborhood of $J$ in $Q(J)$.  Let $\{x_n\}$ be a sequence in $
Q(J)$ converging to a point $z\in J$.  We can choose $n_*$ sufficiently large
that $x_k\gg e_0$ for all $k\ge n_*$, because $z\gg e_0$.  For such
$k$ we have $\ga{x_k}\gg e_0$ by strong monotonicity, whence $x_k
\not\in V^- (J)$.  Similarly for $V^+ (J)$. \epr

\begin{proposition}		\mylabel{th:pnbd4}
$J$ is an attractor for $S|V (J)$.
\end{proposition}
\bpr
By Lemma \rth{pnbd3}  it is enough to prove that
$\lim_{n\to\infty}S^n x=0$ uniformly for $x\in B (J)$. This follows
from  \rth{leaves}(c).
\epr

\begin{lemma}		\mylabel{th:sparcs}
Let $J$ be a stationary p-arc. Suppose $x\in X$
and $\om x \cap J_\bullet \ne\emptyset$.
Then either $\om x \ll J$ or $\om x$ is a singleton in $J$.
\end{lemma}
\begin{proof}
Suppose there exists  $p\in \om x \cap J$.
Then $\om x \cap B (J)\subset {\cal L} (p)$, for otherwise $\om x$
would contain two points of $J$, contradicting $\om x$ being
unordered.  Therefore  $p\in \partial J$, for otherwise $\om x \cap J$
would be a proper  attractor in $\om x$ by 
Propositions \rth{pnbd} and  \rth{pnbd4}.
Thus  $\om x \cap J =\{p\}$ where $p$ is an
endpoint of $J$; we need consider only the case $p=\inf J$.  Then
$\om x \cap  {\cal L} (p)_\bullet = p$.  For otherwise $\om x$ would
contain a point $<p$, by Lemma \rth{leaves}(c).  It
follows that $p$ is an isolated fixed point in $\om x$, which implies
$\om x =\{p\}$.  

Suppose $\om x \cap J$ is empty.  There exists a minimal $q\in J$ such
that $q>\om x$.  Then $q\gg \om x$, whence minimality of $q$ yields
$\om x \ll J$.
\epr

\paragraph{Consequences of chaining.}
In the remainder of this section $T:X\to X$ denotes any strongly
monotone continuous map.  (Here $X$ could be any 
metric space endowed with a closed partial
order relation $R\subset X\times X$.  The notation $x\ll y$ then means 
$(x,y)\in\Int R$.)  

Let  $v,u\in X$ denote chain recurrent
points such that $v>u$.

\begin{proposition}
\mylabel{th:basic}
  Suppose there exists $z\in X$  such that
\[v>Tz > z \ge u\]
and
\[\dist (T (z_+),\: X\setminus \Int (z_+))=\eps>0\]
Then $v\not \leadsto u$. In fact, if $0<\delta<\eps$ and 
$v=x_0,x_1,\ldots,x_n$
is a $\delta$-chain, then $x_n\gg u$.
\end{proposition}
\bpr
$Tv\in \Int T(z_+)$, because
$Tv\gg T^2(z)\gg Tz$ by strong monotonicity. 
This implies 
$x_1\gg z$, i.e., $x_1\in \Int (z_+)$, because:
\beqarr
\dist(x_1, X\setminus \Int (z_+))	 &\ge&
-d(x_1, Tv)\,+\,\dist(Tv,\,X \setminus \Int (z_+))\\
&\ge& -\del\,+\,\dist (T(u)_+,X \setminus \Int (z_+))\\ 
&=& -\del +\eps >0
\eeqarr
Thus
$x_1\gg z\ge u,$ 
and  the same calculation shows by induction on $n$
that all $x_n \gg u$. 
\epr

\begin{corollary}		\mylabel{th:nccor1a}
If $v,u$ are chain equivalent and $v\ge y\ge u$, then $y$ does not
converge eventually monotonically.
\end{corollary}

\begin{proposition}		\mylabel{th:nccor4}
Suppose $x,y$ belong to an attractor-free compact invariant set
$M\subset X$ and $ x <y$. Then $x$ is convergent if and only if
$y$ is convergent.
\end{proposition}
\bpr Assume $y$, but not $x$, is convergent; set $\om y=\{q\}$.  Then
 $\om x \le q$, and in fact $\om x < q$ because $\om x$ is unordered
 and not a singleton.  The set
$N={\om x}_+ \cap M$
is compact, forward invariant, nonempty because $q\in N$,
and a proper nonempty subset of $M$ because $\om x\not\subset N$.
Strong monotonicity shows that $S (N)\subset \Int_M N$, implying that
$N$ contains an attractor for $S|M$; contradiction.  \epr

\section{Proof of Theorem \rth{main} for the case $m=1$} \secl{prfm1}
In this section we assume Hypothesis \rth{hyp0} with $m=1$.

\begin{lemma}		\mylabel{th:lem0}
Let $e\in\Epp$.  Then 
\[\dist (e+ \Ep,\;E\setminus \Epp) >0.\]
\end{lemma}
\bpr Let $x\in\Ep$ be arbitrary
  and choose $y\in E\setminus \Epp$ so that
 that  
\[\dist (e+x,\, E\setminus \Epp)=||e+x-y||=\dist (e,y-x)\]
  Now $y-x\not\in\Epp$ because $x\in
\Ep$ and $x+ (y-x)=y\not\in\Epp$.  Therefore, setting
\[\alpha=\dist (e,\, E\setminus \Epp)>0\]
we have \[\dist (e+x, E\setminus \Epp)\ge \alpha\]
for all $x\in \Ep$.
\epr

Now suppose $a, b$ are chain equivalent and $a<b$.
\begin{proposition}		\mylabel{th:nccor1}
Suppose $a\le y\le b$ and $S^n y$ is
comparable to $S^{n+1}y$ for some $n\ge 0$.  Then $S^n y=S^{n+1}y$.
Thus no orbit entering in $[a,b]$ is eventually  monotone.
\end{proposition}
\bpr Arguing by contradiction, we suppose $S^n y <S^{n+1} y$.  Setting
$S^n y=z,\: u=S^{n+1}a,\: v=S^{n+1}b$, we have
\[u\le z\ll Sz\le v.\] 
From Lemma \rth{lem0} with $e= Sx-x$ we see that 
\[\dist (T (z_+), X\setminus \Int\, z_+)>0\]
 Proposition \rth{basic} gives the contradiction that $u, v$ are
not chain equivalent. 
\epr

\begin{lemma}		\mylabel{th:pr1}
Let $a<x<y<b$ with $\om x=\{p\},\, \om y=\{q\}$.
Then
\[\om a\ll p \ll q \ll \om b,\]
there is a
unique p-arc $J$ with endpoints $p$ and $q$, and any such p-arc
is stationary.
\end{lemma}
\bpr  By monotonicity,
\[\om a \le p \le q \le b\]
Assume $\om a \not \ll p$; then $\om a=\{p\}$ or $\om a <p$.  But the
latter entails $\om a\ll p$ by strong monotonicity and invariance of
$\om a$, so necessarily $\om a = \{p\}$.  Now Theorem \rth{mier1}
implies that either $\ga a$ or $\ga x$ converges monotonically,
contradicting Proposition \rth{nccor1}.  This proves $\om a \ll p$,
and similar arguments prove $p \ll q \ll \om b$.

Theorem \rth{parcs} yields a p-arc joining $p$ to $q$.  By Lemma
\rth{nccor1a}, no orbit in $J$ can converge monotonically.  As $J$ is
totally ordered and invariant, it follows that  $J$ is
stationary.  Uniqueness of $J$ follows easily from strong monotonicity.
 \epr

\begin{lemma}		\mylabel{th:pr1a}
Let $a\le u <v\le b$.  Then there is a stationary p-arc $J$  such that
$\om u  \ll J \ll \om v$.
\end{lemma}
\bpr Choose convergent points $x,y$ such that $Su\ll x \ll y \ll Sv$
(by strong monotonicity, Theorem \rth{gen} and the assumption $m=1$).
Set $\om x=\{p\}, \om y=\{q\}$ and apply Lemma \rth{pr1}.  \epr

\begin{lemma}		\mylabel{th:pr1b}
Every point of $[a,b]$ is convergent.  If $a\le x<y\le b$ then $\om x
\ll \om y$.
\end{lemma}
\bpr
Suppose for example that $a\le x<b$ and $x$ is not convergent, so that
$\om x$ is a 
compact unordered invariant set containing more than one point.  By
Lemma \rth{pr1a} and
the compactness assumption (Hypothesis \rth{hyp0}(v)) there is a
minimal fixed point $p$ satisfying 
\[\om x \ll p \le \om b\]
Pick any $u\in \om x, v \in \om b$.  Then $u\ll p\le v$ and $u$ is
chain equivalent to $v$.  From Lemma \rth{pr1a} we find a fixed point
$q, u \ll q \ll p$.  As this contradicts minimality of $p$, it follows
that $x$ is convergent.  The last sentence is a consequence of Lemma
\rth{pr1a}.  \epr

In the rest of this section we assume the attractor-free set $K$ of
Theorem \rth{main} is not unordered.  Therefore by Lemma \rth{pr1b} we
can select $p\ll q \in K$ with the following properties:

\begin{itemize}
\item $p$ and $q$ are respectively maximal and minimal fixed points in
$K$

\item Every point of $[p,q]$ is convergent

\item If $u<v$ in $[p,q]$ then the trajectories of $u$ and $v$ converge to
distinct fixed points.

\end{itemize}

\begin{lemma}		\mylabel{th:pr5}
$p$ is a minimal point of  $K$ and $q$ is a maximal point of $K$.
\end{lemma}
\bpr
Suppose there exists $u\in K, u<p$.  Then Lemma \rth{pr1a} yields a
stationary p-arc $J$ such that 
\[\om u\ll J \ll \om p=\{p\}\]
But $\om u$ is a singleton, contradicting  minimality of $p$.
\epr

\begin{lemma}		\mylabel{th:pr4}
The set of fixed points in $[p,q]$ is a stationary p-arc $G$ with
endpoints $p, q$.
\end{lemma}
\bpr By Lemma \rth{pr1a} there is a stationary p-arc $G$ joining $p$
to $q$.  Choose $x\in [p,q]\setminus G$; we show $Sx\ne x$.  There is
a minimal $y\in G$ such that $y>x$; then $y\not\gg x$.  Since $y=Sy\gg
Sx$, it follows that $Sx\ne x$.  \epr

Let $B (G)\subset X$ be a contracting collar for $G$ (see Theorem
\rth{parcs}).  I claim $B(G)\cap K$ is a neighborhood of $G\cap K$ in
$K$.  If not, there is a sequence $\{x_n\}$ in $K\setminus B(G)$
converging to an endpoint $p\in G$, by Lemma \rth{pnbd3}.  We assume
$p=\inf G$.  There exists $y\in {\cal L} (p)$ and $k$ such that $x_k <
y$.  Now $\om y=p$ by Lemma \rth{leaves}(c).  Therefore $\om {x_k} \ll
p$ by Theorem \rth{mier1}.  But this contradicts Lemma \rth{pr5}.

It now follows from Lemma \rth{pnbd4} that $G\cap K$ is an attractor
for $S|K$.  Since $K$ is attractor-free, we have proved $K=G\cap K$.
Since $K$ is an attractor-free set of stationary points, it is
connected.  Thus $K$ is a stationary p-arc, showing that either (a) or (b) of
Theorem \rth{main} holds when $m=1$.

Now assume $S$ is real analytic and dissipative.  We show there cannot
be a stationary p-arc $J$.  If there is, by Zorn's lemma there exists
a set $L\subset X$ that is a connected, totally ordered set of
stationary points containing $J$, and which is setwise maximal in
these properties.  Then $L$ is compact because $S$ is dissipative,
whence $L$ is an arc (Wilder \cite{Wi49}).  But there can be no
totally ordered compact arc of fixed points when $S$ is real analytic
(Jiang \& Yu \cite{JY96}, Lemma 3.3 and Theorem 2).  Therefore Theorem
\rth{main}(a) holds.

\section{Proof of Theorem  \rth{main} for the case $m>1$}
\mylabel{sec:red}  
 Assume $m>1$.  Pick an arbitrary $a\in K$ and set 
$L(a)=\Omega (a, S^m|K)$.  This compact subset of $K$
is attractor-free for $S^m$, and it can be shown that
\[S^i L(a)=L (S^i a)=L (S^{i+m} (a))\]
Because $S|K$ is chain transitive, 
$K=\bigcup_{0\le i<m}S^i L(a)$.  
Moreover if $S^i L(a)$ and $S^j L(a)$ intersect, they coincide.

With $a$ chosen once and for all, set 
$K_i= S^i L(a), 0\le i \le m-1$.  Then
$K_i= S^i K_0$, and $K=\bigcup_{0\le i<m}K_i$.

Suppose there exist $x\in K_j, y\in K_k$ with $x<y, j\ne k$; we will
reach a contradiction.  
Relabel the $K_i$ so that $j=0$; then $1\le k\le m-1$, and $y= S^k u,
u\in K_0$.  Therefore the set
\[M=\{x\in K_0: :\:\:\exists\:u\in K_0,
k\in\{1,\ldots,m-1\}\quad\mbox{with}\quad 
x<S^n y\}\] 
is nonempty. $M$ is invariant under $S^m$ and relatively open in $K_0$
by strong monotonicity.  On the other hand, one can also prove $M$
compact.  Therefore $M$ is an attractor for $S^m|K_0$, which implies
$M=K_0$. 

Compactness and strong monotonicity now imply there exists a smallest
$k, 1\le k <m$ such that  $S^k K_0\gg K_0$.
   Therefore
\[K_0=S^{m-k}S^k K_0\gg S^{m-k}K_0\]
Thus $k\le m-k$, and an induction leads to the absurdity that $1\le
b\le k-lm$ for all $l  \in \Np$.  Therefore $M$ is empty, proving that
no points of different $K_i$ are be comparable.  

Since we proved Theorem \rth{main} for $S^m$, it follows that each
$K_i$ is either a totally ordered arc of fixed points for $S^m$,
or else it is unordered.  In both cases the conclusion of Theorem
\rth{main} for $S$ follows.

\begin{remark}		\mylabel{th:rem1}
The proof of Theorem \rth{main} contains the following result for chain
transitive sets that are not necessarily internally chain transitive:

\medskip
\noindent
{\em Let $S$ satisfy Hypothesis \rth{hyp0}, with generic asymptotic period
$m\ge 1$.  Let $L\subset E$ be a compact chain transitive set. Given
$a<b$ in $L$, there is a stationary p-arc $J$ for $S^m$ 
such that if $a\le x \le b$ then $\om x = \ga p$ for some $p\in J$.}

\medskip
\noindent
I suspect, but cannot prove, that  Theorem \rth{main}
is not true under the weaker hypothesis that $K$ is merely chain
transitive.  
\end{remark}

\section{Proof of Theorem \rth{BH}} \secl{prfBH}
 (Adapted from Bena\"{\i}m and Hirsch \cite{BH97}.)  We assume $K$ is
not an upper attracting $m$-periodic orbit.  Consider the case $m=1$.
We first prove that if $x\in K^\bullet \setminus K$, then there is a
fixed point $p$ such that
\[\om x \ge \{p\}\gg K.\] 
Replacing $x$ by  $Sx$, 
we assume $x\gg y\in K$. By Theorem
\rth{gen} we choose $z\in [[y,x]]$ such that $\om z$ is a singleton
$\{p\}$.  Then
\[ \om x\ge p\ge \om y.\]  

I claim $p\not\in K$.  For if $p\in K$  then $\om y =\{p\}$ because $K$ is
unordered, and therefore 
$\ga z \searrow p$ by Theorem \rth{mier1}.  Let $n\ge 0$ be such that
$S^n z \gg p$.  Thus $p$ is an upper attracting fixed point.  Set 
\[N=\{w\in K: w\le S^nz\}\]
 Strong monotonicity implies that $\om w=\{p\}$ for all $w\in N$,
and also that $S$ maps the compact set $N$ into its relative interior
in $K$.  This implies $N$ contains an attractor for $S|K$.  Therefore
$N=K$, yielding the contradiction that that $K=\{p\}$.   Thus $p\not\in K$.

The set $C=\{u\in K:u< p\}$ is a nonempty, forward invariant compact
set.  Strong monotonicity implies $S$ maps $C$ into its relative
interior in $K$, which is $\Int_K C=\{u\in K:u\ll p\}.$ Since $K$ is
attractor-free, $C=K$.  Therefore $\om x \ge p \gg K$.  

 The compactness assumption Hypothesis \rth{hyp0}(v) implies there is a
minimal fixed point $q\gg K$.
 Suppose 
$x\in 
\lfloor K,p_1\rceil \setminus K.$
  We saw
above that $\ga x$ converges to a fixed point $\gg K$.  Since $\ga x
\le q$ by monotonicity, $\om x =\{q\}$.  Taking $x$ sufficiently close
to $K$ proves $q$ lower attracting. 

Suppose $m>1$.  Then $K$ has the partition 
\[K=K_1\cup\cdots \cup K_d\]
where $d|m$, each $K_i$ is attractor-free for $S^m$ and $S$ cyclically
permutes the $K_i$.  Since the generic asymptotic period for $S^m$ is
$1$, the case already proved yields a fixed point  $q$ for $S^m$ 
satisfying  \rth{BH} for $S^m$.  It is easy to see this $q$ 
satisfies Theorem \rth{BH}. \qed

\section{Proofs of theorems on semiflows} \secl{prfsemi}
The proof of Theorem \rth{semi} is almost the same as that of Theorem
\rth{main}.  In place of Tere\v{s}\v{c}\'ak's theorem \rth{gen} one
uses the following result:

\begin{theorem}	[\sc Smith-Thieme \cite{ST91a}]	\mylabel{th:gen2}
There is a dense open set of points whose trajectories converge to
 equilibria.
\end{theorem}
There are also close analogues of Theorems \rth{mier1} and
\rth{parcs}.  The real analytic case is based on a
result in Jiang \cite{Ji91} that rules out setwise maximal totally
ordered stationary p-arcs; see also Chow \& Hale \cite{CH82}, p. 321.
These results are put together just as in the proof of Theorem
\rth{main}.  The details are left to the reader.  \qed

\paragraph{Proof of Theorem \rth{BHsemi}.}
The proof is similar to the proof in Section \rsec{red} of the case
$m=1$ Theorem 
\rth{BH}.   One shows using Theorem \rth{gen2} that if
$x>y\in K$, there is an equilibrium $q_1$ such that
\[\om x \ge q_1 \gg K\]  In this way one shows there exist a smallest
equilibrium $q\gg K$ and this $q$ satisfies the theorem.  \qed

\bigskip
\noindent
\footnotesize {\bf Acknowledgments. }  Correspondence with
M. Bena\"{\i}m and J. Mierczy\'nski has been very helpful in the
preparation of this article.  This research was partially supported by
grants from the National Science Foundation and the North Atlantic
Treaty Organization.  

\renewcommand\refname{\normalsize \bf  References}

\end{document}